\documentclass[11pt]{amsart}
\usepackage{hyperref}
\usepackage{xypic}
\usepackage{color}
\usepackage{latexsym}
\usepackage{amssymb}
\usepackage{lscape}
\usepackage[utf8x]{inputenc}
\usepackage{epsfig,url}
\usepackage[all,cmtip]{xy}
\usepackage[all]{xy}
\usepackage{amsmath}
\usepackage{stmaryrd}
\theoremstyle{plain} % text will be italics
\newtheorem{theorem}{Theorem}[section]
\newtheorem{proposition}[theorem]{Proposition}

\newtheorem{conjecture}[theorem]{Conjecture}

\theoremstyle{definition} % text will be roman

\newtheorem{example}[theorem]{Example}
\newtheorem{convention}[theorem]{Convention}

\newtheorem{remark}[theorem]{Remark}
\newtheorem{definition}[theorem]{Definition}

\DeclareMathOperator{\Id}{Id}

\DeclareMathOperator{\Spec}{Spec}

\DeclareMathOperator{\triv}{\mathbf{1}}

\DeclareMathOperator{\Pure}{\mathbf{Pure}}

\DeclareMathOperator{\rank}{rank}

\DeclareMathOperator{\cone}{cone}

\DeclareMathOperator{\No}{N}

\DeclareMathOperator{\Image}{Image}

\DeclareMathOperator{\Vect}{\mathbf{Vect}}

\DeclareMathOperator{\Mat}{Mat}

\DeclareMathOperator{\id}{id}

\DeclareMathOperator{\HH}{\mathcal{H}}

\DeclareMathOperator{\Fr}{Fr}
\newcommand{\QQ}{\mathbb{Q}}
\DeclareMathOperator{\tw}{\mathtt{tw}}

\DeclareMathOperator{\Rep}{\mathbf{Rep}}

\DeclareMathOperator{\2d}{\mathtt{2d}}
\DeclareMathOperator{\fr}{fr}

\DeclareMathOperator{\KK}{\mathrm{K_0}}
\DeclareMathOperator{\tr}{\mathop{tr}}
\DeclareMathOperator{\Gl}{GL}

\DeclareMathOperator{\fSym}{fSym}

\DeclareMathOperator{\Iso}{Iso}
\DeclareMathOperator{\Jac}{\mathbf{Jac}}

\DeclareMathOperator{\Betti}{\mathtt{Betti}}
\DeclareMathOperator{\Dolb}{\mathtt{Dol}}
\DeclareMathOperator{\hp}{\mathtt{hp}}
\DeclareMathOperator{\EE}{E}
\newcommand{\Ho}{\mathrm{H}}
\newcommand{\HO}{\mathcal{H}}

\newcommand{\spec}{\mathop{\rm Spec}\nolimits}

\DeclareMathOperator{\Exp}{\mathbf{Exp}}

\newcommand{\Coh}{\mathop{\rm Coh}\nolimits}

\renewcommand{\phi}{\varphi}

\newcommand{\Hom}{\mathop{\rm Hom}\nolimits}

\newcommand{\Ext}{\mathop{\rm Ext}\nolimits}
\newcommand{\End}{\mathop{\rm End}\nolimits}

\newcommand{\Db}[1]{\mathrm{D}^b(#1)}
\newcommand{\Dbf}[1]{\mathrm{D}^b_{fd}(#1)}
\newcommand{\Dbc}[1]{\mathbf{D}_c^b(#1,\QQ)}

\newcommand{\GL}{{\rm GL}}

\newcommand{\PGL}{{\rm PGL}}

\newcommand{\MHM}{\mathop{\mathbf{MHM}}\nolimits}

\newcommand{\Gr}{\mathop{\rm Gr}\nolimits}

\newcommand{\pt}{\mathop{\rm pt}}

\renewcommand{\tilde}{\widetilde}
\newcommand{\Cp}{\mathbb{C}}

\newcommand{\M}{\mathcal M}

\renewcommand{\O}{\mathcal O}

\newcommand{\x}{{\bf x}}

\renewcommand{\v}{{\bf v}}

\def\co{\colon\thinspace}

\begin{document}
\author[B. Davison ]{Ben Davison}
\address{B. Davison: EPFL}
\email{nicholas.davison@epfl.ch}

\title[Cohomological Hall algebras and character varieties]{Cohomological Hall algebras and character varieties}
\maketitle
\begin{abstract}
In this paper we investigate the relationship between twisted and untwisted character varieties via a specific instance of the Cohomological Hall algebra for moduli of objects in 3-Calabi-Yau categories introduced by Kontsevich and Soibelman. In terms of Donaldson--Thomas theory, this relationship is completely understood via the calculations of Hausel and Villegas of the E polynomials of twisted character varieties and untwisted character stacks. We present a conjectural lift of this relationship to the cohomological Hall algebra setting.
\end{abstract}
\section{Introduction}
A fundamental object of research in the study of Higgs bundles on a genus\footnote{Throughout the paper we assume $g\geq 1$.} $g$ complex curve is the \textit{twisted character variety}
\begin{equation}
\label{BettiSpace}
\M^{\Betti,\tw}_{g,n}:=\left\{\begin{tabular}{c}$A_1,\ldots,A_g,B_1,\ldots,B_g\in\GL_{n}(\Cp)$\textrm{ such that}\\ $\prod_{i=1}^{g}A_iB_iA_i^{-1}B_i^{-1}=\exp(2\pi \sqrt{-1}/n)\Id_{n\times n}$\end{tabular}\right\}/\PGL_n(\Cp),
\end{equation}
where the set in brackets is considered as a sub algebraic variety of $\GL_n(\Cp)^{2g}$, and the action of $\PGL_n(\Cp)$ is the simultaneous conjugation action on all of the $A_i$ and the $B_i$.  By the quotient we mean the categorical quotient in the category of complex schemes -- this exists by a theorem of Nagata, see for instance \cite{Ne78}.  By \cite[Cor.2.2.7]{HRV08} the $\PGL_n(\Cp)$-action is free, so that the underlying topological space of (\ref{BettiSpace}) is the orbit space of the $\PGL_n(\Cp)$-action.
\smallbreak
The link between (\ref{BettiSpace}) and Higgs bundles is as follows.  Let $C$ be a nonsingular complete complex genus $g$ curve.  The moduli space of semistable rank $n$ degree $1$ Higgs bundles on $C$ is defined as follows:
\begin{equation}
\label{DolbSpace}
\M^{\Dolb}_{n,1}(C):=\left\{\begin{tabular}{c}semistable $(V,\theta)$ with $V$ a vector bundle, \\$\theta\in \Ho^0\big(\Omega_C\otimes\End(V)\big)$, $\deg(V)=1$ and $\rank(V)=n$\end{tabular}\right\}/\textrm{isomorphism}.
\end{equation}
The $\theta$ in the above definition is known as the Higgs field, an isomorphism of pairs $(V,\theta)\rightarrow (V',\theta')$ is an isomorphism $f:V\rightarrow V'$ compatible with $\theta$ and $\theta'$ in the obvious way, and semistability is the condition that any sub-bundle $V'$ of $V$ preserved by the Higgs field satisfies
\[
\deg(V')/\rank(V')\leq 1/n.
\]
The quotient (\ref{DolbSpace}) arises as a complex algebraic variety via Geometric Invariant Theory -- see for example \cite{Si94}.  The nonabelian Hodge theorem, proved for complex curves via a combination of work of Hitchin \cite{Hi87}, Corlette \cite{Co88} and Donaldson \cite{Do87}, states that there is a diffeomorphism
\[
\Phi\co\M^{\Betti,\tw}_{g,n}\cong\M^{\Dolb}_{n,1}(C)
\]
and so we obtain an isomorphism
\[
\Ho^{\bullet}(\Phi,\QQ)\co\Ho^{\bullet}\big(\M^{\Dolb}_{n,1}(C),\QQ\big)\cong\Ho^{\bullet}\big(\M^{\Betti,\tw}_{g,n},\QQ\big)
\]
between the singular cohomology groups with rational coefficients of these two varieties (the reader may wish to consult the appendix of \cite{We08} for a very approachable overview of this part of the theory).  The mixed Hodge structure on $\Ho^{i}(\M^{\Dolb}_{n,1}(C),\QQ)$ is pure of weight $i$ (see for example the discussion in \cite{deCa11}); on the other hand the mixed Hodge structure on $\Ho^{i}(\M^{\Betti,\tw}_{g,n},\QQ)$ is \textit{not} pure.  This leads to some of the main problems in the subject: to understand the mixed Hodge structure $\Ho^{\bullet}(\M^{\Betti,\tw}_{g,n},\QQ)$, and to describe the image of the weight filtration of $\Ho^{\bullet}(\M^{\Betti,\tw}_{g,n},\QQ)$ on $\Ho^{\bullet}(\M^{\Dolb}_{n,1}(C),\QQ)$ under the isomorphism $\Ho^{\bullet}(\Phi,\QQ)$.  The P=W conjecture of \cite{CaHaMi12} states that the weight filtration on $\Ho^{\bullet}(\M^{\Betti,\tw}_{g,n},\QQ)$ becomes the \textit{perverse filtration} on $\Ho^{\bullet}(\M^{\Dolb}_{n,1}(C),\QQ)$, defined in terms of the Hitchin system (see \cite{CaHaMi12} for details).  We will concentrate on $\Ho^{\bullet}(\M^{\Betti,\tw}_{g,n},\QQ)$, saying a few words about the P=W conjecture at the end of the paper.
\smallbreak
Consider instead the space
\begin{equation}
\label{BettiStack}
\M^{\Betti}_{g,n}:=\left\{\begin{tabular}{c}$A_1,\ldots,A_g,B_1,\ldots,B_g\in\GL_n(\Cp)$ \\such that $\prod_{i=1}^g A_iB_iA_i^{-1}B_i^{-1}=\Id_{n\times n}$\end{tabular}\right\}/\GL_n(\Cp).
\end{equation}
The space in brace brackets can be considered as a variety parametrising representations of $\pi_1(C)$.  If $n\geq 2$, then in contrast with the twisted character variety (\ref{BettiSpace}) the action of $\GL_n(\Cp)$ on the space in brackets is not free, even after replacing it with the action of $\PGL_n(\Cp)$ -- for instance if $\triv$ denotes the trivial representation of $\pi_1(C)$, the stabiliser group of $\triv^{\oplus n}$ is the whole of $\PGL_n(\Cp)$.  We consider (\ref{BettiStack}) as a stack theoretic quotient -- it is isomorphic to the Artin stack of $n$-dimensional representations of $\pi_1(C)$.  An overview of Artin stacks, and in particular global quotient stacks, is provided by Gomez's paper \cite{Go01}.  The definition of the stack (\ref{BettiStack}) as a functor from affine schemes to groupoids starts as follows: $X=\Spec(R)$ is sent to the groupoid of $R\otimes\Cp[\pi_1(C)]$-modules which are locally free of rank $n$ over $R$.
\smallbreak
The space $\M^{\Betti,\tw}_{g,n}$ is smooth, and as such there is an isomorphism of mixed Hodge structures
\begin{equation}
\label{PDiso}
\Ho^{\bullet}\big(\M^{\Betti,\tw}_{g,n},\QQ\big)\cong\Ho^{\bullet}_c\big(\M^{\Betti,\tw}_{g,n},\QQ\big)^*\{\dim_{\Cp}(\M^{\Betti,\tw}_{g,n})\}
\end{equation}
given by Poincar\'e duality.  Here and from now on we use the notation $L\{i\}:=L\otimes_{\mathbb{Q}}\mathbb{Q}(-i)[-2i]$, where $\mathbb{Q}(-i)$ is the 1-dimensional mixed Hodge structure of weight $2i$ and the square brackets denote the cohomological shift of degree.  On the other hand the \textit{untwisted character stack} $\M^{\Betti}_{g,n}$ is not smooth, so it matters whether we study its cohomology or dual compactly supported cohomology.  We pick the latter, and recall the definition in Section \ref{Tbackground}. 
\smallbreak
To relate $\Ho^{\bullet}(\M^{\Betti,\tw}_{g,n},\QQ)$ to $\Ho_c^{\bullet}(\M^{\Betti}_{g,n},\QQ)^*$ we introduce some notation and results from \cite{HRV08}.  Given a cohomologically graded mixed Hodge structure $L^{\bullet}$, define the mixed Hodge polynomial
\[
\hp(L^{\bullet},x,y,t)=\sum_{a,b,j\in\mathbb{Z}}\dim_{\Cp}\big(\Gr_a^F(\Gr_{a+b}^W(L^j\otimes_{\mathbb{Q}}\Cp))\big)x^ay^bt^j.
\]
All mixed Hodge structures $L$ that arise will satisfy $\Gr_{a}^F(\Gr^W_{a+b}(L\otimes_{\mathbb{Q}}\Cp))=0$ if $a\neq b$, so we may as well pass to the two variable specialization $\hp(L^{\bullet},q,t)=\hp(L^{\bullet},\sqrt{q},\sqrt{q},t)$.  Setting $\EE(L^{\bullet},q):=\hp(L^{\bullet},q,-1)$ we obtain a specialization of $\hp(\Ho_c^{\bullet}(X,\QQ))$, for $X$ an algebraic variety, that is a motivic invariant in the sense that for $U\subset X$ an open subvariety, $\EE(\Ho_c^{\bullet}(X,\QQ))=\EE(\Ho_c^{\bullet}(U,\QQ))+\EE(\Ho_c^{\bullet}(X\setminus U,\QQ))$.  The main result of \cite{HRV08} is the explicit calculation of $\EE(\Ho_c^{\bullet}\big(\M^{\Betti,\tw}_{g,n},\QQ\big))$.  Using the same techniques the authors also calculate $\EE(\Ho_c^{\bullet}\big(\M^{\Betti}_{g,n},\QQ\big))$ \cite[Thm.3.8.1]{HRV08}.
\smallbreak
To relate these two calculations we use the language of \textit{plethystic exponentials} (see \cite{Ge95}).  Let $\Vect_{\mathbb{Z}^l}$ be the category of $\mathbb{Z}^l$-graded vector spaces, with finite dimensional graded pieces.  Taking characteristic polynomials gives an isomorphism $\KK(\Vect_{\mathbb{Z}^l})\xrightarrow{\chi}\mathbb{Z}[[x_1^{\pm 1},\ldots,x_l^{\pm 1}]]$.  Let $\Vect_{\mathbb{Z}^I}^{+}\subset\Vect_{\mathbb{Z}^l}$ be the subcategory of vector spaces $V$ which are strictly positively graded with respect to the first $\mathbb{Z}$-grading, and for which the coefficient of each $x_1^a$ in $\chi([V])$ is a formal function in the remaining variables with a finite order pole at the origin.  We may identify $\KK(\Vect_{\mathbb{Z}^l}^+)$ with a subring of $\mathbb{Z}[[x_1^{\pm 1},\ldots,x_l^{\pm 1}]]$ via $\chi$.  Furthermore there is a functor $\fSym\co\Vect_{\mathbb{Z}^l}^+\rightarrow \Vect_{\mathbb{Z}^l}$ taking $V$ to the underlying vector space of the free commutative algebra generated by $V$, and a function 
\[\Exp:=\KK(\fSym)\co \KK(\Vect_{\mathbb{Z}^l}^+)  \rightarrow \mathbb{Z}[[x_1^{\pm 1},\ldots,x_l^{\pm 1}]].\]
\smallbreak
%To illuminate a little what this function is, consider the case of a monomial $z=x_1^{a_1}\ldots x_l^{a_l}\in x_1\mathbb{Z}[[x_1,x_2^{\pm 1},\ldots,x_l^{\pm 1}]]$.  We have $z=\chi(V)$ where $V$ is a one dimensional vector space concentrated in degree $(a_1,\ldots,a_l)$.  Let $w$ be a basis element for $V$, then 
%\begin{align*}
%\Exp(z)=&\Exp(\chi(V))\\=&\chi(\fSym(V))\\=&\chi(\Cp[w])\\=&1+z+z^2+\ldots.  
%\end{align*}
%Now consider the case of $-z$.  We can think of this as $\chi(V[1])$, the characteristic function of $V$, placed in odd cohomological degree.  By the Koszul sign rule, cohomologically odd elements of a supercommutative algebra square to zero so that we have $\Sym(V[1])\cong \Cp\oplus V[1]$ and so $\Exp(-z)=1-z$.
%\smallbreak
%The functor $\fSym$ above is the composition of two functors: the functor $\Sym:\Vect_{\mathbb{Z}^l}\rightarrow \Alg(\Vect_{\mathbb{Z}^l})$ taking a graded vector space to the free supercommutative algebra generated by it, and the functor $\forget:\Alg(\Vect_{\mathbb{Z}^l})\rightarrow \Vect_{\mathbb{Z}^l}$ that forgets the multiplication.  The algebra $\Sym(V)$ is the enveloping algebra of the Lie algebra $V$ equipped with the trivial Lie bracket.  However one can just as well calculate $\fSym(V)$ by taking $\forget(U(\mathfrak{g}))$ for $\mathfrak{g}$ any Lie algebra with underlying vector space $V$.
\smallbreak
Returning to character varieties, we relate the calculation of $\EE(\Ho_c^{\bullet}(\M^{\Betti,\tw}_{g,n},\QQ))$ to $\EE(\Ho_c^{\bullet}(\M^{\Betti}_{g,n},\QQ))$.  Consider the graded mixed Hodge structures 
\begin{align}
\label{untwstco}
\HO^{\Betti}_g:=&\bigoplus_{n\geq 0}\Ho_c(\M^{\Betti}_{g,n},\QQ)\{(1-g)n^2\}^*,\\
\label{twstco}
\HO^{\Betti,\tw}_g=&\bigoplus_{n\geq 1}\Ho_c([\M^{\Betti,\tw}_{g,n}/\Gl_{1}(\Cp)],\QQ)\{(1-g)n^2\}^*.
\end{align}
Here we give $\M^{\Betti,\tw}_{g,n}$ the trivial $\Gl_{1}(\Cp)$ action.  There is an isomorphism in cohomology
\[
\Ho_c([\M^{\Betti,\tw}_{g,n}/\Gl_{1}(\Cp)],\QQ)\cong \Ho_c\left(\left[\left\{\begin{tabular}{c}$A_1,\ldots,A_g,B_1,\ldots,B_g\in\GL_{n}(\Cp)$\textrm{ such that}\\ $\prod_{i=1}^{g}A_iB_iA_i^{-1}B_i^{-1}=\exp(2\pi \sqrt{-1}/n)\Id_{n\times n}$\end{tabular}\right\}/\GL_n(\Cp)\right],\QQ\right)
\]
so that (\ref{twstco}) is the cohomology of the stack of twisted representations of the fundamental group of $C$.  Applying $\EE(\bullet)$ to each graded piece of these mixed Hodge structures, we obtain formal power series $\EE(\HO^{\Betti,\tw}_g):=\sum_{n\geq 1} \EE(\HO^{\Betti,\tw}_{g,n})x^n\in x\mathbb{Z}[[x,q^{\pm 1}]]$ and $\EE(\HO^{\Betti}_g)\in\mathbb{Z}[[x,q^{\pm 1}]]$.  Combining the results of \cite{HRV08} gives the remarkable relation
\begin{equation}
\label{TheRelation}
\Exp(\EE(\HO_g^{\Betti,\tw}))=\EE(\HO_g^{\Betti}).
\end{equation}
The goal of this paper is to understand relation (\ref{TheRelation}).  We will show that, guided by the theory of BPS algebras, or cohomological Hall algebras of objects in 3-Calabi-Yau categories, we can put a kind of Hopf algebra structure on the mixed Hodge structure $\HO^{\Betti}_g$, and we conjecture that the resulting algebra satisfies a PBW theorem.  The relation (\ref{TheRelation}) then becomes the statement that the $\EE$ series of the subspace of primitive elements in the PBW basis for $\HO^{\Betti}_g$ is exactly the $\EE$ series of $\HO^{\Betti,\tw}_g$, leading naturally to our main conjecture.
\begin{conjecture}
\label{mainconj1}
There is an isomorphism of mixed Hodge structures
\[
\fSym(\HO_g^{\Betti,\tw})\cong\HO_g^{\Betti}.
\]
\end{conjecture}
The conjecture implies that the mixed Hodge polynomials of the spaces $\M_{g,n}^{\Betti,\tw}$ are encoded in the mixed Hodge power series of the spaces $\M_{g,n}^{\Betti}$, and vice versa, providing new lines of attack on the conjectures of Hausel and Rodriguez-Villegas regarding these polynomials.
\section{Equivariant Cohomology and Vanishing Cycles}
\label{Tbackground}
Before launching into the construction of cohomological Hall algebras arising from Jacobi algebras and vanishing cycles, we collect together some of the background definitions.

Let $X$ be an arbitrary variety equipped with a $G$-action, and a faithful embedding of algebraic groups $G\subset \Gl_{\Cp}(m)$.  We first define the compactly supported cohomology of the \textit{global quotient stack} [X/G] (see \cite{Go01} for a definition of this stack).  For $M\geq m$ let $\Fr(m,M)$ be the space of $m$-tuples of linearly independent vectors in $\Cp^M$.  There are natural inclusions $\Fr(m,M)\rightarrow\Fr(m,M+1)$ inducing inclusions
\[
X\times_{G}\Fr(m,M)\xrightarrow{i_M} X\times_{G}\Fr(m,M+1)
\]
and Gysin morphisms in the category of mixed Hodge structures
\begin{equation}
\label{capp}
\Ho_c\big(X\times_{G}\Fr(m,M),\QQ\big)\{-mM\}\rightarrow\Ho_c\big(X\times_{G}\Fr(m,M+1),\QQ\big)\{-m(M+1)\}
\end{equation}
obtained by applying $\left(X\times_G \Fr(m,M+1)\rightarrow \pt\right)_!$ to the adjunction map 
\begin{equation}
\label{adjmap}
i_{M,!}\mathbb{Q}_{X\times_G \Fr(m,M)}\{-mM\}\rightarrow \mathbb{Q}_{X\times_G \Fr(m,M+1)}\{-m(M+1)\}.
\end{equation}
We define $\Ho_c([X/G],\QQ)$ to be the limit of these maps.
\smallbreak
Now let $X$ be a smooth complex variety, and let $f\in \Gamma(\O_X)$ be a function on it.  Let $\Dbc{X}$ denote the derived category of sheaves of $\QQ$-vector spaces on $X$ with analytically constructible cohomology (all subsequent functors are assumed to be derived).  Then if we define $X_{>0}:=\{x\in X|f(x)\in \mathbb{R}_{>0}\}$ and $X_0:=\{x\in X|f(x)=0\}$, we define the functor $\psi_f\co\Dbc{X}\rightarrow\Dbc{X}$ by
\[
\psi_f=(X_0\rightarrow X)_*(X_0\rightarrow X)^*(X_{>0}\rightarrow X)_*(X_{>0}\rightarrow X)^*.
\]
For instance, applying $\psi_f$ to $\mathbb{Q}_X$, we obtain the sheaf
\[
(X_0\rightarrow X)_*(X_0\rightarrow X)^*(X_{>0}\rightarrow X)_*\QQ_{X_{>0}}
\]
supported on $X_0$, the \textit{sheaf of nearby cycles} on $X$.  As defined this is actually an object in $\Dbc{X}$, and is rarely represented by an actual sheaf.  Via the adjunction $\id\rightarrow (X_{>0}\rightarrow X)_*(X_{>0}\rightarrow X)^*$ we obtain a natural transformation
\begin{equation}
\label{adjmap}
(X_0\rightarrow X)_*(X_0\rightarrow X)^*\xrightarrow{q} \psi_f
\end{equation}
and we define $\phi_f=\cone(q)$ (with some care this cone can be made functorial -- see for instance Ex. VIII.13 of \cite{KS90}).  By abuse of notation we will often just denote $\phi_f:=\phi_f\mathbb{Q}_{X}[-1]$.  The shift here is for book-keeping purposes later.  In fact the functor $\phi_f$ defined above lifts to an endofunctor of the derived category of mixed Hodge modules $\phi_f\co\Db{\MHM(X)}\rightarrow \Db{\MHM(X)}$.  For this paper we needn't say anything about the category of mixed Hodge modules except that there is a forgetful functor $\Db{\MHM(X)}\rightarrow \Dbc{X}$ which is faithful, the six functor formalism of Grothendieck and the functors $\psi_f$ and $\phi_f$ lift to $\Db{\MHM(X)}$, and $\Db{\MHM(\pt)}$ is the derived category of mixed Hodge structures.  The interested reader can consult \cite{Sai89} for more details.  We recover Deligne's mixed Hodge structure on $\Ho_c^{\bullet}(X,\QQ)$ for $X$ an arbitrary variety, by applying $(X\rightarrow \pt)_!$ to the constant mixed Hodge module $\mathbb{Q}(0)$ on $X$, and the mixed Hodge structure of Steenbrink and Navarro Aznar on $\Ho^{\bullet}(X,\phi_f)$ by applying $(X\rightarrow \pt)_*$ to $\phi_f\in\Db{\MHM(X)}$.  We will use four facts regarding vanishing cycles:
\begin{enumerate}

\item
For $p\co X\rightarrow Y$ a proper map and $f\in\Gamma(\mathcal{O}_Y)$, there is a natural isomorphism $\phi_fp_*\rightarrow p_*\phi_{fp}$.
\label{propcomm}
\item
\label{AffFib}
For $q\co X\rightarrow Y$ an $n$-dimensional affine fibration, there is a natural isomorphism $\Ho_c^{\bullet}(Y,\phi_{f})\cong \Ho_c^{\bullet}(X,\phi_{fq})\{-{n}\}$.
\item
\label{degloc}
The support of $\phi_f\QQ_X$ is exactly the degeneracy locus of $f$.  By shrinking $X$ we always assume that this is a subspace of $f^{-1}(0)$.
\item{(Thom--Sebastiani isomorphism)}
\label{HTS}
For $f_i\in\Gamma(\mathcal{O}_{Y_i})$ two functions, there is a natural isomorphism $\Ho_{c}(Y_1,\phi_{f_1})\otimes\Ho_c(Y_2,\phi_{f_2})\cong\Ho_c(Y_1\times Y_2,\phi_{\pi_1^*f_1+\pi_2^*f_2})$.
%\item
%\label{fact4}
%There is an isomorphism $\phi_f\mathbb{Q}_X\cong D\phi_f\mathbb{Q}_X\{\dim_{\mathbb{C}}(X)\}$, so that there is an isomorphism $\Ho_c^{\bullet}(\phi_f)^*\{\dim_{\mathbb{C}}(X)\}\cong\Ho(\phi_f)$.
\end{enumerate}
The fourth fact is a theorem of Massey \cite{Ma01}, at the level of the underlying cohomologically graded vector spaces, and an unpublished theorem of Saito at the level of `monodromic mixed Hodge structures'.  For the mixed Hodge structures we will encounter, there is an independent proof from the theory of dimensional reduction, see the appendix of \cite{Chicago2}.

Finally, let $X$ be a smooth algebraic variety equipped with a $G$-action, where as above we have a faithful embedding $G\subset \Gl_{\mathbb{C}}(m)$, and let $f$ be a $G$-invariant function on $X$.  Then $f$ induces functions $f_M$ on each of the spaces $X\times_G \fr(m,M)$, and applying $\left(X\times_G \fr(m,M+1)\rightarrow \pt\right)_!\phi_{f_{M+1}}$ to the adjunction map (\ref{adjmap}), and applying the natural isomorphism of fact (\ref{propcomm}) above, we obtain Gysin maps 
\begin{equation}
\label{betterap}
\Ho_c\big(X\times_{G}\Fr(m,M),\phi_{f_M}\big)\{-mM\}\rightarrow\Ho_c\big(X\times_{G}\Fr(m,M+1),\phi_{f_{M+1}}\big)\{-m(M+1)\}
\end{equation}
and we define $\Ho_c([X/G],\phi_f)$ to be the limit of these maps.
\smallbreak

\section{The Theory of BPS State Counting for 3--Calabi--Yau Categories}
\label{CoHAconst}
The algebra structure that we define on $\HO^{\Betti}_g$ comes from the Hall algebra construction in noncommutative 3-Calabi-Yau geometry introduced by Kontsevich and Soibelman in \cite{CoHA}.  We provide a short overview of the theory, in the generality that we need.
\begin{definition}\cite{KontRos00}
A (not necessarily commutative) algebra $B$ is \textit{nc smooth} if it is finitely generated and for any algebra $A$, and any two-sided ideal $I\lhd A$ satisfying $I^n=0$ for $n\gg 0$, every algebra homomorphism $f:B\rightarrow A/I$ lifts to a morphism $\tilde{f}:B\rightarrow A$ such that the composition $B\rightarrow A\rightarrow A/I$ is equal to $f$.
\end{definition}
Given a finitely generated algebra $B$ we define $\Rep_n(B)$ to be the stack of $n$-dimensional representations of $B$.  In the language of sheaves of groupoids, if $A$ is a commutative algebra, $\Rep_n(B)(\spec(A))$ is the groupoid obtained by forgetting noninvertible morphisms in the category of $A\otimes B$-modules, locally free over $A$, of rank $n$ at each geometric point of $A$.  
\begin{convention}
\label{liftconv}
It will often be convenient to fix a presentation for $B$:
\[
B\cong\Cp\langle x_1,\ldots,x_t\rangle/\langle r_1,r_2,\ldots\rangle.
\]
Let $\overline{\Rep}_n(B)\subset \Mat_{n\times n}(\Cp)^{\times t}$ be the subscheme cut out by the matrix valued relations $r_1,r_2,\ldots$.  Then $\Rep_n(B)$ is isomorphic to the stack theoretic quotient $[\overline{\Rep}_n(B)/\Gl_{n}(\Cp)]$ formed by equipping $\Mat_{n\times n}(\Cp)^{\times t}$ with the simultaneous conjugation action.  From this description we see that $\Rep_n(B)$ is a finite type global quotient Artin stack.  If $f$ is a function on $\Rep_n(B)$ we denote by $\overline{f}$ the induced function on $\overline{\Rep}_n(B)$.
\end{convention}
\begin{proposition}
Let $B$ be a nc smooth algebra.  Then $\Rep_n(B)$ is a finite type smooth Artin stack.
\end{proposition}
\begin{proof}
We have already seen that $\Rep_n(B)$ is a finite type Artin stack.  It is enough to show that the atlas $\overline{\Rep}_n(B)$ is smooth (see \cite{Go01}).  Recall the following criterion for smoothness \cite{DiGr64}: a finite type scheme $Y$ is smooth if every map $\Spec(R/I)\rightarrow Y$, for $I$ a nilpotent ideal of $R$ an Artinian local ring, can be lifted to a map $\Spec(R)\rightarrow Y$.  A map $\Spec(R/I)\rightarrow \overline{\Rep}_n(B)$ is given by a map $B\rightarrow \Mat_{n\times n}(R/I)\cong \Mat_{n\times n}(R)/\Mat_{n\times n}(I)$.  Now a lift exists since $\Mat_{n\times n}(I)$ is a 2-sided nilpotent ideal of $\Mat_{n\times n}(R)$.
\end{proof}
\begin{example}
\label{Qex}
Let $Q$ be a quiver with vertices $Q_0$ and arrows $Q_1$ (our quivers are always assumed to satisfy $|Q_0|,|Q_1|<\infty$).  Let $B=\Cp Q$ be the free path algebra of $Q$.  Then $B$ is a nc smooth algebra, and 
\[
\Rep_n(B)\cong\sum_{\substack{n_1,\ldots,n_{|Q_0|}\geq 0\\ \sum n_i=n}}\left[\prod_{a\in Q_1}\Hom(\Cp^{n_{s(a)}},\Cp^{n_{t(a)}})/\prod_{i\in Q_0}\Gl_{n_i}(\Cp)\right],
\]
is a finite type smooth stack (here the $\prod_{i\in Q_0}\Gl_{n_i}(\Cp)$-action is via change of basis on the $\Cp^{n_i}$).
\end{example}
\begin{remark}
As we see from Example \ref{Qex}, the stack of $n$-dimensional representations of a quiver $Q$, and hence the stack of representations of $\Cp Q/I$ for any two-sided ideal $I$, breaks naturally into a disjoint union indexed by $\gamma\in\mathbb{N}^{Q_0}$ with $\sum\gamma_i=n$.  We define $\Rep_{\gamma}(\Cp Q/I)$ to be the substack corresponding to the dimension vector $\gamma$.
\end{remark}
\begin{example}
\label{lqex}
Let $Q$ be a quiver, and let $Q'\subset Q$ be a subquiver.  For each arrow $a\in Q'_1$ add an arrow $a^*$ with $s(a^*)=t(a)$ and $t(a^*)=s(a)$ to form a new quiver $\overline{Q}$ (here and elsewhere $s$ and $t$ stand for source and target).  Recall that in the path algebra $\mathbb{C} Q$, the symbol $e_i$ denotes the path of zero length at the vertex $i$.  Then we define 
\[
\widetilde{\Cp Q}:=\Cp \overline{Q}/\langle a^*a=e_{s(a)},aa^*=e_{t(a)}|a\in Q'_1\rangle, 
\]
the localized path algebra.  Using the previous example one can see that this algebra is smooth, and
\[
\Rep_{\gamma}(\widetilde{\mathbb{C} Q})\cong\left[\left(\prod_{a\in Q_1\setminus Q'_1}\Hom(\Cp^{\gamma_{s(a)}},\Cp^{\gamma_{t(a)}})\times\prod_{a\in Q'_1}\Iso(\Cp^{\gamma_{s(a)}},\Cp^{\gamma_{t(a)}})\right)/\prod_{i\in Q_0}\Gl_{\gamma_i}(\Cp)\right],
\]
is a Zariski open substack of $\Rep_{\gamma}(\Cp Q)$.  We use the notation $\Iso(V',V'')$ to denote isomorphisms from a vector space $V'$ to a vector space $V''$.  The stack $\Rep_{\gamma}(\widetilde{\mathbb{C} Q})$ is a dense open substack of $\Rep_{\gamma}(\Cp Q)$ if and only if for every $a\in Q'$, $\gamma_{s(a)}=\gamma_{t(a)}$.
\end{example}
Given $W$ (called a \textit{potential}) in the vector space quotient $B/[B,B]$ we obtain a function $\tr(W)_{\gamma}$ on $\Rep_{\gamma}(B)$ as follows.  First, lift $W$ to an element $\tilde{W}\in B$.  For a representation $\rho$ of $B$, we obtain an element $\tr(\rho(\tilde{W}))$, independent of which lift $\tilde{W}$ we choose, by cyclic invariance of the trace.  It follows that $\rho\rightarrow \tr(\rho(\tilde{W}))$ defines a function on $\Rep_{\gamma}(B)$.
\smallbreak
Given the pair $(B,W)$ of a smooth noncommutative algebra with potential, one forms as in \cite{ginz} the Jacobi algebra $\Jac(B,W)$.  We will restrict to the case in which $B=\tilde{\Cp Q}$ is the localized path algebra associated to a pair $Q'\subset Q$, and $W\in \Image \left(\Cp{Q}/[\Cp Q,\Cp Q]\rightarrow B/[B,B]\right)$ is a linear combination of cyclic paths in $Q$ -- this simplifies the definition of $\Jac(B,W)$.  Given an arrow $a\in Q_1$, and $u$ a cyclic path in $Q$, we define 
\[
\partial{u}/\partial{a}=\sum_{\substack{u=vaw\\v,w\textrm{ paths in }Q}}wv
\]
and we extend to a function $\partial/\partial{a}:\Cp Q/[\Cp Q,\Cp Q]\rightarrow B$ by linearity.  Then
\[
\Jac(B,W):= B/\langle \partial W/\partial a|a\in Q_1\rangle.
\]
\begin{example}
\label{3loopex}
Let $Q$ be the quiver with one vertex and three loops, labelled $x,y,z$, and let $Q'=Q$.  Then the localized path algebra $\tilde{\Cp Q}$ is $\Cp\langle x^{\pm 1},y^{\pm 1},z^{\pm 1}\rangle$, the Laurent polynomial algebra in three noncommuting variables.  Let $W=x[y,z]$.  Then 
\begin{align*}
\partial W/\partial x=&[y,z]\\
\partial W/\partial y=&[z,x]\\
\partial W/\partial z=&[x,y]
\end{align*}
and $\Jac(\tilde{\Cp Q},W)\cong\Cp[x^{\pm 1},y^{\pm 1},z^{\pm 1}]$, the commutative Laurent polynomial algebra in three variables.  Note that $\Jac(\tilde{\Cp Q},W)\cong\Cp[\pi_1((S^1)^3)]$, the fundamental group algebra of the 3-torus.  We will see with Proposition \ref{Gtrue} that some other fundamental group algebras of 3-manifolds arise as Jacobi algebras.
\end{example}
We consider $\Rep_{\gamma}(\Jac(B,W))$ as a substack of $\Rep_{\gamma}(B)$ in the natural way: the relations $\partial W/\partial a$ define matrix valued functions on $\Rep_{\gamma}(B)$, and $\Rep_{\gamma}(\Jac(B,W))$ is the stack theoretic vanishing locus of these functions.  Alternatively:
\begin{proposition}{\cite[Sec.2.3]{ginz} \cite[Prop.3.8]{Seg08}}:
\label{JacDeg}
$\Rep_{\gamma}(\Jac(B,W))\subset\Rep_{\gamma}(B)$ is the stack-theoretic degeneracy locus of $\tr(W)_{\gamma}$.
\end{proposition}
\smallbreak
From now on we assume that $B$ is the (possibly localized) path algebra of a quiver $Q$.  For $W\in \mathbb{C}Q/[\mathbb{C}Q,\mathbb{C}Q]$, by fact (\ref{degloc}) from Section \ref{Tbackground} and Proposition \ref{JacDeg}, $\phi_{\tr(W)_{\gamma}}$ may be considered as an object of $\Dbc{\Rep_{\gamma}(\Jac(B,W))}$, although we have defined it as an object of $\Dbc{\Rep_{\gamma}(B)}$.  We define the mixed Hodge structure\footnote{In the sequel $-{\dim_{\Cp}(\Rep_{\gamma}(B))/2}$ will always be an integer.  See \cite[Sec.3.4]{CoHA} for the general case.}
\begin{align}
\label{degshift}
\HH_{B,W,\gamma}:=&\Ho_c^{\bullet}\left(\Rep_{\gamma}(\Jac(B,W)),\phi_{\tr(W)_{\gamma}}\right)\{-{\dim_{\Cp}(\Rep_{\gamma}(B))/2}\}^*\\=&\Ho_c^{\bullet}\left(\Rep_{\gamma}(B),\phi_{\tr(W)_{\gamma}}\right)\{-{\dim_{\Cp}(\Rep_{\gamma}(B))/2}\}^*\nonumber
\end{align}
and 
\[
\HH_{B,W}:=\bigoplus_{\gamma\in\mathbb{N}^{Q_0}}\HH_{B,W,\gamma}.
\]
\smallbreak
Next we recall the algebra structure on $\HH_{B,W}$.  For $\gamma'+\gamma''=\gamma$ we define $\Rep_{\gamma',\gamma''}(B)$ to be the stack of pairs $\rho_1\subset\rho_2$, where $\rho_1$ is a $\gamma'$-dimensional representation of $B$ and $\rho_2$ is a $\gamma$-dimensional representation.    We may describe this stack as in Example \ref{lqex}:
\[
\Rep_{\gamma',\gamma''}(B)\cong\left[\left(\prod_{a\in Q_1\setminus Q'_1}\Hom^p(\Cp^{\gamma_{s(a)}},\Cp^{\gamma_{t(a)}})\times\prod_{a\in Q'_1}\Iso^p(\Cp^{\gamma_{s(a)}},\Cp^{\gamma_{t(a)}})\right)/\prod_{i\in Q_0}P_{\gamma'_i,\gamma''_i}(\Cp)\right],
\]
where $P_{\gamma'_i,\gamma''_i}(\Cp)\subset \Gl_{\gamma_i}(\Cp)$ is the subgroup preserving $\Cp^{\gamma'_i}$, and $\Hom^p$ and $\Iso^p$ are the subspaces of $\Hom$ and $\Iso$ preserving the flags $\Cp^{\gamma'_i}\subset \Cp^{\gamma_i}$.  There is a diagram
\[
\xymatrix{
\Rep_{\gamma'}(B)\times \Rep_{\gamma''}(B)&\ar[l]_-q\Rep_{\gamma',\gamma''}(B)\ar[r]^-i&\Rep_\gamma(B)
}
\]
and the idea of the multiplication on $\HH_{B,W}$ is to pull back compactly supported cohomology along the map $i$, then push forward along the map $q$, then dualize.  In a little more detail, we have seen that $\Rep_{\gamma',\gamma''}(B)$ is the stack theoretic quotient of an affine scheme $X$ by $P=\prod_{i\in Q_0}P_{\gamma'_i,\gamma''_i}$, while $\Rep_{\gamma}(B)$ is the stack theoretic quotient of an affine scheme $Y$ by $G=\prod_{i\in Q_0}\Gl_{\gamma_i}(\Cp)$.  The compactly supported cohomology $\Ho_{c}([X/P],\phi_{\tr(W)_{\gamma}})$ is approximated (as in (\ref{betterap})) by $\Ho_c(X\times_P\Fr(n,N),\phi_{\tr(W)_{\gamma,N}})$, where $n=\sum\gamma_i$, and $\Ho_c([Y/G],\phi_{\tr(W)_{\gamma}})$ is defined similarly.  The inclusion $X\rightarrow Y$ is $P$-equivariant, and we have a proper composition of maps
\[
X\times_P\Fr(n,N)\xrightarrow{i_1} Y\times_P\Fr(n,N)\xrightarrow{i_2} Y\times_G\Fr(n,N).
\]
Applying the functor $\phi_{\tr(W)_{\gamma,N}}$ to the adjunction $\mathbb{Q}_{Y\times_G\Fr(n,N)}\rightarrow i_{2,*}i_{1,*}\mathbb{Q}_{X\times_P\Fr(n,N)}$ and using fact (1) regarding vanishing cycles we obtain a map
\[
\Ho_{c}(Y\times_G\Fr(n,N),\phi_{\tr(W)_{\gamma,N}})\rightarrow \Ho_{c}(X\times_P\Fr(n,N),\phi_{\tr(W)_{\gamma,N}})
\]
which gives the desired map $\Ho_c([Y/G],\phi_{\tr(W)_{\gamma}})\rightarrow \Ho_c([X/P],\phi_{\tr(W)_{\gamma}})$ in the limit.  The pushforward along $q$ is defined in much the same way, this time expressing the map $q$ in terms of affine fibrations, and using fact (\ref{AffFib}) regarding vanishing cycles.
\smallbreak
Via the above constructions and the Thom--Sebastiani isomorphism one obtains an associative product $m\co \HH_{B,W}\otimes\HH_{B,W}\rightarrow \HH_{B,W}$ (see \cite{CoHA} for more details).  This is the \textit{cohomological Hall algebra} (CoHA) associated to $(B,W)$.
\begin{remark}
\label{preservation}
One may consider the subalgebra that is the sum of those $\mathcal{H}_{B,W,\gamma}$ for which $\gamma$ satisfies the condition that if $\gamma_i\neq \gamma_j$, there are as many arrows in $Q$ from $i$ to $j$ as from $j$ to $i$.  For this subalgebra the multiplication $m$ preserves the cohomological grading, and is a morphism of mixed Hodge structures.
\end{remark}
\begin{remark}
Again imposing the restriction of Remark \ref{preservation}, the CoHA $\mathcal{H}_{B,W}$ carries a richer structure, making it into a kind of Hopf algebra object in the derived category of mixed Hodge structures.  For a $G$-variety $X$ there is a natural morphism $[X/G]\rightarrow [X/G]\times [\pt/G]$, inducing the structure of a $\Ho^{\bullet}([\pt/G],\QQ)\cong\Ho^{\bullet}_G(\pt,\QQ)$-module on $\Ho_c^{\bullet}([X/G],\phi_f)^*$, for $f$ a $G$-invariant function on $X$.  There is a localized coproduct $\mathcal{H}_{B,W,\gamma}\rightarrow \left(\mathcal{H}_{B,W,\gamma'}\otimes \mathcal{H}_{B,W,\gamma''}\right)\otimes_{A_{\gamma',\gamma''}} \widetilde{A_{\gamma',\gamma''}}$, where $A_{\gamma',\gamma''}:=\Ho^{\bullet}_{G_{\gamma'}}(\pt,\QQ)\otimes\Ho^{\bullet}_{G_{\gamma''}}(\pt,\QQ)$, and $\widetilde{A_{\gamma',\gamma''}}$ is the localization of $A_{\gamma',\gamma''}$ at the equivariant Euler class of a specific virtual bundle on $\Rep_{\gamma'}(B)\times\Rep_{\gamma''}(B)$.  This makes $\mathcal{H}_{B,W}$ into a localized Hopf algebra (see \cite{Chicago2}).
\end{remark}

\section{The link between Character Varieties and BPS State Counting}
The link between the cohomology of character varieties and BPS state counting, or CoHAs, comes in two steps.  Firstly we describe a class of Jacobi algebras $\Jac(\widetilde{\mathbb{C}Q}_{\Delta},W_{\Delta})$ that arise as noncommutative compactifications of $\Cp[\pi_1(\Sigma_g\times S^1)]$.  Then we use \textit{dimensional reduction}, which gives an isomorphism of mixed Hodge structures $\mathcal{H}_{\widetilde{\mathbb{C}Q}_{\Delta},W_{\Delta},(n,\ldots,n)}\cong \HO^{\Betti}_{g,n}$.
\smallbreak
The first step requires the theory of brane tilings of Riemann surfaces.  A brane tiling $\Delta$ of $\Sigma_g$ is an embedding of a bipartite graph $\Gamma$ in $\Sigma_g$ such that each connected component of $\Sigma_g\setminus \Gamma$ (or `tile') is simply connected.  We assume that we are given a colouring of the vertices of $\Gamma$ with black and white such that no two vertices of the same colour are joined by an edge.  We pick a smooth embedding of the dual graph of $\Delta$ in $\Sigma_g$, this is the underlying graph of the quiver $Q_{\Delta}$, which we orient so that for each black vertex $v$ of $\Delta$ the edges of $Q_{\Delta}$ that are dual to edges containing $v$ form a clockwise cycle.  For every vertex $v\in V(\Delta)$ there is a minimal cyclic path $c_v$ of $Q_{\Delta}$ containing all the duals of the edges containing $v$, and we define
\[
W_{\Delta}=\sum_{v\textrm{ white}}c_v-\sum_{v\textrm{ black}}c_v.
\]
In this way we obtain an algebra $\Jac(\mathbb{C}Q_{\Delta},W_{\Delta})$.  Let $e$ be an edge of $\Delta$, joining vertices $u$ and $v$, dual to the arrow $a\in (Q_{\Delta})_1$.  The expression $\partial W/\partial a$ is the difference $c_v'-c_u'$, where $c_v'$ and $c_u'$ are obtained by cyclically permuting $c_v$ so that $a$ is at the front, and then deleting it.  In other words, the relation $\partial W/\partial a=0$ imposes the condition that two homotopic paths in the quiver $Q_{\Delta}$ become equal in the Jacobi algebra $\Jac(\mathbb{C}Q_{\Delta}, W_{\Delta})$.  See Fig. \ref{Fterm} for an illustration.  This does not however give a bijection between isoclasses of paths from $i$ to $j$ in $Q_{\Delta}$ under the equivalence relation given by equality in $\Jac(\mathbb{C}Q_{\Delta},W_{\Delta})$ and homotopy classes of paths between the vertices $i$ and $j$ in $\Sigma_g$; for instance the path $c_u$ is contractible, but is not equal to the path $e_{s(c_u)}$ in $\Jac(\mathbb{C}Q_{\Delta},W_{\Delta})$.
\begin{figure}
\centering
\input{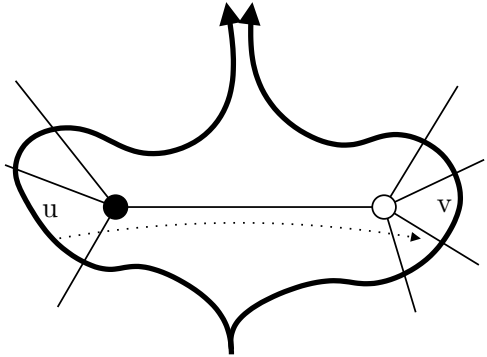}
\caption{Equivalent paths in $\Jac(\mathbb{C}Q_{\Delta},W_{\Delta})$ are homotopic.}
\label{Fterm}
\end{figure}
\smallbreak
We obtain a better picture by thinking of paths in $Q=Q_{\Delta}$ as moving in three dimensions.  For this, assume that the edges of $Q$ are graded by the numbers $\mathbb{Q}$ in such a way that $W_{\Delta}$ is homogeneous of weight $1$ -- we then say that $\Delta$ is a graded brane tiling.  Since $Q$ is already embedded in $\Sigma_g$, we may define an embedding $\iota\co Q_0\rightarrow \Sigma_g\times S^1$ by defining a function $\pi_{S^1}\iota\co Q_0\rightarrow \mathbb{R}/\mathbb{Z}$.  First pick $v_0\in Q_0$ and define $\pi_{S^1}\iota(v_0)=[0]$.  For arbitrary $v\in Q_0$ define $\pi_{S^1}\iota(v)=[|p|]$ in terms of our $\mathbb{Q}$-grading, where $p$ is any path in $Q$ from $v_0$ to $v$.  This $\pi_{S^1}\iota$ is well-defined by \cite[Lem.2.8]{Dav08}.  We extend the embedding $\iota\co Q_0\subset \Sigma_g\times S^1$ to a smooth embedding $Q\subset \Sigma_g\times S^1$ which becomes the existing embedding $Q\subset \Sigma_g$ after projection onto $\Sigma_g$.  We assume that the embedding is minimal in the following sense: considering each arrow $a$ as a path $\alpha\co[0,1]\rightarrow \mathbb{R}/\mathbb{Z}$ via the projection onto $S^1$, the derivative of $\alpha$ is nonnegative, and $\int_0^1 \alpha'=|a|$.
\smallbreak
We say that $\Delta$ is \textit{consistent} if for all paths $u_1,u_2$ from $i$ to $j$ in $Q$, if $c_v^nu_1$=$c_v^nu_2$, considered as elements of $\Jac(\mathbb{C} Q_{\Delta},W_{\Delta})$, then $u_1=u_2$ in $\Jac(\mathbb{C} Q_{\Delta},W_{\Delta})$.  Here $c_v$ is a minimal cycle around a vertex in $\Delta$ with $s(c_v)=j$.  If we consider $\Jac(\mathbb{C} Q_{\Delta},W_{\Delta})$ in the natural way as an algebra with many objects (the objects being $Q_0$), consistency is equivalent to the condition that $\Jac(\mathbb{C} Q_{\Delta},W_{\Delta})$ is an integral domain with many objects.  This condition is easy to satisfy, in particular, for every genus $g$ there is a consistent, graded brane tiling for $\Sigma_g$ -- see \cite{Dav08} for more discussion.  The relevance of consistency here is the following.
\begin{proposition}[\cite{Dav08}]
If $\Delta$ is consistent then two paths $u_1,u_2$ in $Q$ from $i$ to $j$ are equal in $\Jac(\mathbb{C}Q_{\Delta},W_{\Delta})$ if and only if they are homotopic in $\Sigma_g\times S^1$.
\end{proposition}
Pick a maximal tree $T\subset Q_{\Delta}$, and let $(\Sigma_g\times S^1)'$ be obtained by contracting $\iota(T)$ to a point.  The natural map $\pi_1((\Sigma_g\times S^1))\rightarrow \pi_1((\Sigma_g\times S^1)')$ is an isomorphism.  The above proposition states that if $\Delta$ is consistent then the resulting map $e_i\Jac(\mathbb{C}Q_{\Delta},W_{\Delta})e_j\rightarrow\Cp[\pi_1((\Sigma_g\times S^1)')]$ is injective for all $i,j\in (Q_{\Delta})_0$, and we obtain an embedding depending on our choice of $T$
\begin{equation}
\label{tdef}
t_{\Delta}\co\Jac(\mathbb{C}Q_{\Delta},W_{\Delta})\subset \Mat_{r\times r}(\Cp[\pi_1(\Sigma_g\times S^1)]),
\end{equation}
where $r=|(Q_{\Delta})_0|$.
\begin{proposition}
\label{Gtrue}
Let $\Delta$ be an arbitrary brane tiling of $\Sigma_g$, and localize with respect to the pair $Q_{\Delta}\subset Q_{\Delta}$.  Then there is an isomorphism
\[
\Jac(\tilde{\Cp Q_{\Delta}},W_{\Delta})\cong\Mat_{r\times r}(\Cp[\pi_1(\Sigma_g\times S^1)]).
\]
\end{proposition}
\begin{proof}
We extend $t_{\Delta}$ to a map $\Jac(\tilde{\Cp Q_{\Delta}},W_{\Delta})\rightarrow\Mat_{r\times r}(\Cp[\pi_1(\Sigma_g\times S^1)])$.  It is not hard to check that this map is surjective; we show injectivity.  If $t_{\Delta}(p)=t_{\Delta}(p')$ for two paths $p$ and $p'$ in $\overline{Q_{\Delta}}$ starting at $i\in (Q_{\Delta})_0$, then $t_{\Delta}(p^{-1}p')=e_{ii}$, the elementary matrix with entry $1$ in the $(i,i)$th place and zeros elsewhere.  But $p^{-1}p'=c_i^\lambda$ in $\Jac(\tilde{\Cp Q_{\Delta}},W_{\Delta})$ for some $\lambda$, by \cite[Lem.2.7]{Dav08}, from which we deduce that $\lambda=0$ and $p=p'$ in $\Jac(\widetilde{\mathbb{C} Q_{\Delta}},W_{\Delta})$.
\end{proof}
Note that no reference to consistency is made in Proposition \ref{Gtrue}.  Thinking of $\Jac(\tilde{\Cp Q_{\Delta}},W_{\Delta})$ as an algebra with many objects, this is explained by the fact that the localized Jacobi algebra $\Jac(\tilde{\Cp Q_{\Delta}},W_{\Delta})$ is always an integral domain with many objects, regardless of whether $\Delta$ is consistent or not.

%Note that in the above proposition we do not have to assume that $\Delta$ is consistent. 
\section{Dimensional Reduction for Brane Tiling Algebras}
Recall that we always assume that $g\geq 1$.  Let $\Delta$ be a consistent brane tiling on $\Sigma_g$, and let $\Dbf{\Jac(\Cp Q_{\Delta},W_{\Delta})}$ be the subcategory of the derived category of complexes of modules over the associated Jacobi algebra consisting of complexes with finite dimensional total cohomology.  Then $\Dbf{\Jac(\Cp Q_{\Delta},W_{\Delta})}$ shares some features with $\Db{\Coh(X)}$, the bounded derived category of coherent sheaves on a smooth projective Calabi--Yau 3-fold $X$, namely by \cite[Cor.4.4]{Dav08} there is a natural equivalence of bifunctors
\[
\Ext^i(M,N)\simeq \Ext^{3-i}(N,M)^*
\]
for $N,M\in \Dbf{\Jac(\Cp Q_{\Delta},W_{\Delta})}$.  On the other hand, by Poincar\'e duality, $\Dbf{\Cp[\pi_1(\Sigma_g)]}$ carries a similar equivalence of bifunctors, but with a shift of 2 -- the category $\Dbf{\Cp[\pi_1(\Sigma_g)]}$ has more in common with the category of coherent sheaves on a Calabi--Yau 2-fold than a 3-fold.  The purpose of this section is to explain how it is, then, that the cohomology of the untwisted character stack shows up in the study of CoHAs for certain brane tiling algebras.
\smallbreak
We use cohomological dimensional reduction of vanishing cycles, for which the setup is very general (see \cite{BBS} for the motivic analogue).  Let $Y$ be a $G$-variety, and let $E$ be the total space of a $G$-equivariant vector bundle on $Y$.  We assume that every point of $Y$ is contained in a $G$-equivariant affine subvariety of $Y$.
\begin{proposition}
\label{dimredprop}
Let $f$ be a regular function on $[E/G]$ that is homogeneous of weight one with respect to the scaling action of $\Cp^*$ along the fibres of $\pi\co E\rightarrow Y$.  Let $i\co[Z/G]\hookrightarrow f^{-1}(0)$ be the maximal subspace of $f^{-1}(0)$ satisfying $\pi^{-1}\pi([Z/G])=[Z/G]$.  Then there is a natural equivalence
\begin{equation}
\label{NatEq}
\Upsilon\co\pi_!\phi_f\pi^*[-1]\cong \pi_!i_*i^*\pi^*\co\Db{\MHM([Y/G])}\rightarrow\Db{\MHM([Y/G])}.
\end{equation}
\end{proposition}
The proof in the case in which $Y$ is a single affine variety is contained in the appendix of \cite{Chicago2}.  The general case is then a consequence of the fact that the statement is local on $Y$.  Applying $([Y/G]\rightarrow \pt)_!$ to $\Upsilon(\mathbb{Q}_{[Y/G]})$, we obtain
\begin{align*}
&\Ho_c^{\bullet}([E/G],\phi_f)\cong\Ho_c^{\bullet}([Z/G],\QQ)&\big(\cong\Ho_c^{\bullet}([\pi(Z)/G],\QQ)\{\dim_{\Cp}(\pi)\}\big).
\end{align*}
Now let $(Q,W)$ be a quiver with potential.  A \textit{cut} of $(Q,W)$ is a set $E\subset Q_1$ of edges such that if we grade the edges $Q_1$ by setting $|a|=1$ if $a\in E$ and $|a|=0$ otherwise, $W$ is homogeneous of degree one.  Let $E\subset Q_1$ be a cut of $(Q,W)$, and let $Q'\subset Q$ be a subquiver containing none of the arrows of $E$.  We define $Q^-=Q\setminus E$ by removing the edges of $E$, and consider the pair $Q'\subset Q^-$, forming the localized path algebra $\tilde{\Cp Q^-}$.  We define the \textit{2-dimensional Jacobi algebra}
\[
\Jac_{\2d}(\tilde{\Cp Q},W,E):=\tilde{\Cp Q}/\langle a,\partial W/\partial a|a\in E\rangle\cong\tilde{\Cp Q^-}/\langle \partial W/\partial a|a\in E\rangle.
\]
Note that although $a\in E$ is not an element of $\tilde{\Cp Q^-}$, $\partial W/\partial a$ is, by the grading conditions.
\begin{example}
Let $(Q,W)$ be the three loop quiver with potential from Example \ref{3loopex}, let $E=\{z\}$, and let $Q'\subset Q$ be the quiver containing the two loops $x,y$.  Then $Q^-=Q'$, and $\tilde{\Cp Q^-}=\Cp\langle x^{\pm 1},y^{\pm 1}\rangle$.  The set $\{\partial W/\partial a|a\in E\}$ is just $\{xy-yx\}$, so that
\[
\Jac_{\2d}(\tilde{\Cp Q},W,E)\cong\Cp[x^{\pm 1},y^{\pm 1}]\cong\Cp[\pi_1((S^1)^2)].
\]
\end{example}
\begin{proposition}
\label{DRJ}
Let $Q'\subset Q,W,E$ satisfy the above conditions. There is an isomorphism of cohomologically graded mixed Hodge structures
\[
\Ho_c^{\bullet}\big(\Rep_{\gamma}(\Jac(\tilde{\Cp Q},W),\phi_{\tr(W)_{\gamma}})\big)\cong\Ho_c^{\bullet}\big(\Rep_{\gamma}(\Jac_{\2d}(\tilde{\Cp Q},W,E)),\QQ\big)\{\prod_{a\in E}\gamma_{s(a)}\gamma_{t(a)}\}.
\]
\end{proposition}
\begin{proof}
Note that $\pi\co\Rep_{\gamma}(\tilde{\Cp Q})\rightarrow \Rep_{\gamma}(\tilde{\Cp Q^{-}})$ is a $\sum_{a\in E}\gamma_{s(a)}\gamma_{t(a)}$-dimensional $\Gl_{\gamma}(\Cp)$-equivariant vector bundle, since we do not localize with respect to any of the arrows in $E$.  So we can apply Proposition \ref{dimredprop}, since $\tr(W)_{\gamma}$ is a weight one function on $\Rep_{\gamma}(\tilde{\Cp Q})$ with respect to the scaling action of $\pi$.  We have to work out $Z$, in the notation of Proposition \ref{dimredprop}.  In the notation of Convention \ref{liftconv}, it is enough to work out $\overline{Z}\subset \overline{\Rep}_{\gamma}(\widetilde{\mathbb{C} Q})$.  We write $\overline{\tr(W)_{\gamma}}=\sum x_{i,j,a}\pi^*(f_{i,j,a})$ where $f_i$ are functions on $\overline{\Rep}_{\gamma}(\tilde{\Cp Q^{-}})$ and the $x_{i,j,a}$ are linear coordinates on the fibres of $\pi$ given on a representation $\rho\in\overline{\Rep}_{\gamma}(\tilde{\Cp Q})$ by the $(i,j)$th entry of $\rho(a)$, for $a\in E$.  Then $Z$ is the locus where all the $f_{i,j,a}$ vanish.  But $f_{i,j,a}$ is just the $(j,i)$th entry of $\partial W/\partial a(\rho)$, so $Z$ is the locus where all the matrix valued functions $\partial W/\partial a$ vanish, i.e. $Z=\pi^{-1}(\Rep_{\gamma}(\Jac_{\2d}(\tilde{\Cp Q},W,E)))$.
\end{proof}

\begin{proposition}
\label{JacMat}
Let $\Delta$ be a brane tiling of $\Sigma_g$, let $E$ be a cut of $Q_{\Delta}$, and let $Q'\subset Q_{\Delta}$ be the subquiver containing all those edges not contained in $E$.
Let $r=|(Q_{\Delta})_0|$.  There is an isomorphism 
\[
\Jac_{\2d}(\tilde{\Cp Q_{\Delta}},W_{\Delta},E)\cong\Mat_{r\times r}(\Cp[\pi_1(\Sigma_g)]).
\]
\end{proposition}
\begin{proof}
There is a map $j\co\Jac_{\2d}(\tilde{\Cp Q_{\Delta}},W_{\Delta},E)\rightarrow \Jac(\tilde{\Cp Q_{\Delta}},W_{\Delta})$ induced by the natural map 
\[
\tilde{\Cp Q_{\Delta}^{-}}/\langle \partial W_{\Delta}/\partial a|a\in E\rangle\rightarrow \Jac(\tilde{\Cp Q_{\Delta}},W_{\Delta}),
\]
this is a map of nonnegatively graded algebras after placing the domain in degree zero and giving the target the grading satisfying $|a|=1$ for $a\in E$, $|a|=0$ otherwise.  The degree zero part of the Jacobi ideal $\langle \partial W_{\Delta}/\partial a|a\in (Q_{\Delta})_1\rangle$ is exactly $\langle \partial W_{\Delta}/\partial a|a\in E\rangle$, considered as a two-sided ideal of $\tilde{\Cp Q_{\Delta}^-}$, and so $j$ is just the inclusion of the degree zero piece of $\Jac(\tilde{\Cp Q_{\Delta}},W_{\Delta})$.
\smallbreak
We extend $t_{\Delta}$ of (\ref{tdef}) to $\Jac(\tilde{\Cp Q_{\Delta}},W_{\Delta})$ -- the extension is injective by the same argument as Prop. \ref{Gtrue}.  Then we restrict this map to $\Jac_{\2d}(\tilde{\Cp Q_{\Delta}},W_{\Delta},E)$.  The image lies in $\Mat_{r\times r}(\Cp[\pi_1(\Sigma_g\times \pt)])$ since the arrows of $Q'$ have weight zero.  We show that this map is surjective.  For this it is enough to show that for any path $p$ between two vertices $i$ and $j$ of $\overline{Q_{\Delta}}$ there is a homotopic path in $\overline{Q'}$.  Denote $a$ by $a^+$ and $a^*$ by $a^-$.  Every arrow $a_0^{\pm}$ contained in $p$ that is not contained in $\overline{Q'}$ is contained in a minimal cycle $c_v=a_0a_1\ldots a_r$.  We can replace $a_0^{\pm}$ by $a_r^{\mp}\ldots a_1^{\mp}$ to obtain a homotopic path, since $c_v$ is contractible.  Note that all of $a_1,\ldots,a_r\in \overline{Q'}_1$ by the condition on $Q'$.
\end{proof}
%The following is standard Morita theory:
\begin{proposition}
\label{Morita}
For a finitely generated algebra $A$ there is an equivalence of Artin stacks
\[
\Rep_m(A)\cong\Rep_{rm}(\Mat_{r\times r}(A)).
\]
\end{proposition}
\begin{proof}
Say a $B\otimes\Mat_{r\times r}(A)$-module $M$ is represented by a map $\Spec(B)\xrightarrow{f} \Rep_{rm}(\Mat_{r\times r}(A))$.  Since $\Rep_{rm}(\Mat_{r\times r}(A))$ is finite type, $f$ factors as $\Spec(B)\xrightarrow{p} \Spec(B')\xrightarrow{g}\Rep_{rm}(\Mat_{r\times r}(A))$ with $p$ corresponding to an inclusion of a Noetherian ring $B'\subset B$, and $g$ corresponding to a $B'\otimes \Mat_{r\times r}(A)$-module $M'$.  Let $C_i$ be the space of $r\times r$ matrices with entries in $A$ which are zero away from the $i$th column.  Since $M'\cong\oplus C_i\cdot M'$ is locally free as a $B'$-module it is projective, since $B'$ is Noetherian.  It follows that each of the summands $C_i\cdot M'$ are projective too, so $C_i\cdot M'$ is locally free, and hence so is $C_i\cdot M\cong p^*C_i\cdot M$.  So $C_i\cdot$ is a natural functor from the groupoid of $B\otimes\Mat_{r\times r}(A)$ modules locally free over $B$ to the groupoid of $B\otimes A$-modules locally free over $B$, with a natural inverse sending $M$ to $M\otimes\Cp^r$ with the natural $B\otimes \Mat_{r\times r}(A)$-action.
\end{proof}
\begin{theorem}
Let $V(\Delta)$ be the number of vertices in a brane tiling $\Delta$ of a surface $\Sigma$.  Assume that $(Q_{\Delta},W_{\Delta})$ admits a cut $E$, and define $Q'\subset Q_{\Delta}$ and $\widetilde{\mathbb{C}Q_{\Delta}}$ as in Proposition \ref{JacMat}.  There is an isomorphism of mixed Hodge structures
\begin{equation}
\label{mainsh}
\Ho_c^{\bullet}\big(\Rep_{(n,\ldots,n)}(\Jac(\tilde{\Cp Q_{\Delta}},W_{\Delta}),\phi_{\tr(W_{\Delta})})\big)\cong \Ho_c^{\bullet}\big(\Rep_n(\mathbb{C}[\pi_1(\Sigma)]),\QQ\big)\{V(\Delta)n^2/2\},
\end{equation}
so the mixed Hodge structure $\HO_g^{\Betti}$ of (\ref{untwstco}) carries the structure of a localized Hopf algebra in the category of cohomologically graded mixed Hodge structures.
\end{theorem}
\begin{proof}
For the first statement, put together Propositions \ref{DRJ}, \ref{JacMat} and \ref{Morita}, and observe that $|E|=V(\Delta)/2$.  For the second we use the construction of Section (\ref{CoHAconst}), noting the degree shift by $\dim_{\mathbb{C}}(\Rep_{(n,\ldots,n)}(\tilde{\mathbb{C}Q}))/2$ in (\ref{degshift}), Remark \ref{preservation}, and the calculation 
\begin{align*}
V(\Delta)n^2-\dim_{\mathbb{C}}(\Rep_{(n,\ldots,n)}(\tilde{\mathbb{C}Q}))=&V(\Delta)n^2-(|Q_1|n^2-|Q_0|n^2)\\
=&(2-2g)n^2.
\end{align*}
See \cite{Chicago2} for a construction of such a $(Q_{\Delta},W_{\Delta})$ for $g\geq 1$.
\end{proof}

\section{Conclusion and further directions}
It is striking that the shift in the relation (\ref{TheRelation}) is exactly the shift required to turn the dual compactly supported cohomology of untwisted character stacks into a CoHA.  It leads naturally to the following conjecture (suggested by Olivier Schiffmann), which would imply Conjecture \ref{mainconj1}:
\begin{conjecture}
There is a Lie algebra structure on $\bigoplus_{n\geq 1}\Ho_c(\M^{\Betti,\tw}_{g,n},\QQ)\{(1-g)n^2\}^*$, and a filtration $Y$ on $\HO^{\Betti}_g$ such that there is an isomorphism of algebras 
\[
\Gr_{\bullet}^Y(\HO^{\Betti}_g)\cong\mathcal{U}\left(\bigoplus_{n\geq 1}\Ho_c(\M^{\Betti,\tw}_{g,n},\QQ)\{(1-g)n^2\}^*[u]\right).
\]
\end{conjecture}
Here $u$ is a formal variable of weight and cohomological degree 2, and we extend the Lie bracket via $[gu^i,g'u^j]=[g,g']u^{i+j}$.  The conjecture is partly motivated by analogy with the case of quiver varieties (see below, and \cite[Thm.5.5.1]{MaOk12}), and partly by relation (\ref{TheRelation}), itself a consequence of Conjecture \ref{mainconj1}.
\smallbreak
We finish by returning to the P=W conjecture.  We say a cohomologically graded mixed Hodge structure $L^{\bullet}$ is \textit{pure} if the $i$th graded piece $L^i$ is pure of weight $i$.  Note that for $k\in\mathbb{Z}$, $L^{\bullet}\{k\}$ is pure if $L^{\bullet}$ is.  Furthermore, since $\Gr^W_j(\Ho_c^{i}(X,\QQ))=0$ for $X$ smooth and $j>i$, it follows from the long exact sequence in compactly supported cohomology that $\Gr^W_j(\Ho^{i}_c(X,\QQ)\{k\})=0$, for all schemes $X$ and all values of $k$ and $j>i$.  Since the multiplication and comultiplication in $\HO_g^{\Betti}$ are morphisms of cohomologically graded mixed Hodge structures, it follows that there is a sub localized bialgebra of pure cohomology $\Pure(\HO_g^{\Betti})\subset \HO_g^{\Betti}$.  From Conjecture \ref{mainconj1}, the conjectural form for the mixed Hodge polynomials of $\HO_{g,n}^{\Betti,\tw}$ given in \cite[Conj.4.2.1]{HRV08}, and \cite[Thm.5.1]{Moz11}, we obtain the following prediction:
\begin{equation}
\label{toppart}
\EE\big(\Pure(\HO_{g,n}^{\Betti})\big)=\EE\big(\Ho_c([\mu_{n,g}^{-1}(0)/\Gl_n(\mathbb{C})],\QQ)\{(1-g)n^2\}^*\big)
\end{equation}
where $\mu_{n,g}:\Mat_{n\times n}(\mathbb{C})^{\times 2g}\rightarrow \Mat_{n\times n}(\mathbb{C})$ is given by $\sum_{j=1}^g [A_j,B_j]$.
\smallbreak
Eq. (\ref{toppart}) illustrates the main point of this paper: Conjecture \ref{mainconj1} gives a way to translate conjectures regarding the cohomology of twisted character varieties into very different, but equivalent, conjectures regarding the cohomology of their untwisted counterparts.  Eq. (\ref{toppart}) is the untwisted cousin of the purity conjecture of \cite[Rem.4.4.2]{HRV08}.  
\smallbreak
Let $x\in X$ be a point of an algebraic variety.  Recall the deformation to the normal cone of $x$ in $X$ (see \cite[Ch.5]{Fu84}): there is a map $f:Y\rightarrow \mathbb{A}_{\mathbb{C}}^1$ such that the pullback along the inclusion $\mathbb{A}_{\mathbb{C}}^1\setminus \{0\}\rightarrow \mathbb{A}_{\Cp}^1$ is the trivial family with fibre $X$, and the normal cone $N_x$ to $x$ embeds into the fibre $Y_0$ as an open subvariety.  It follows that $\Ho_c(Y,\psi_f)\cong\Ho_c(X,\QQ)$, and we have a composition of morphisms $\Ho_c(N_x,\QQ)\rightarrow \Ho_c(Y_0,\QQ)\rightarrow \Ho_c(X,\QQ)$, the first coming from the open inclusion, the second coming from the map (\ref{adjmap}).  This construction gives a natural categorification of (\ref{toppart}) (in other words a map to underly the conjectural equality of generating series) suggested by Tam\'as Hausel.  The normal cone to $\triv^{\oplus n}\in \Rep_n(\Cp[\pi_1(\Sigma_g)])$ is precisely $[\mu_{n,g}^{-1}(0)/\Gl_n(\mathbb{C})]$, giving the map
\begin{equation}
\label{NCmap}
\Psi:\HO_g^{\Betti}\rightarrow \bigoplus_n \Ho_c\big([\mu_{n,g}^{-1}(0)/\Gl_n(\mathbb{C})],\QQ\big)\{(1-g)n^2\}^*.
\end{equation}
\smallbreak
\begin{conjecture}
The map $\Psi$ is an isomorphism after restricting to $\Pure(\HO^{\Betti}_g)$.
\end{conjecture}
So far it is known at least that $\Psi$ is a retraction when $g=1$ -- this is easy to see after observing that there is an open embedding $\Rep_n(\Cp[\pi_1(\Sigma_1)])\subset[\mu^{-1}_{n,1}(0)/\Gl_n(\Cp)]$ that takes $\triv^{\oplus n}$ to the vertex of the cone $[\mu^{-1}_{n,1}(0)/\Gl_n(\Cp)]$.  The cohomological Hall algebra that is the target of $\Psi$ is itself the object of study from various directions, and is the object of interest when making the link with the work of Maulik and Okounkov \cite{MaOk12} and, separately, Schiffmann and Vasserot \cite{ScVa12} on the construction of Yangians associated to Nakajima quiver varieties.
\smallbreak
We finally come to the P=W conjecture.  On the Higgs bundle side, the map (\ref{NCmap}) has a natural analogue, given by deformation to the normal cone of the point $(\mathcal{O}_C^{\oplus n},0)\in\mathcal{M}^{\Dolb}_{n,0}(C)$, the \textit{stack} of degree zero semistable Higgs bundles.  By converting the problem to degree zero (the equivalent move on the Higgs bundle side to moving attention from twisted to untwisted character varieties) we arrive at a conjectural geometric description of the lowest perversity part of the cohomology of the stack of semistable Higgs bundles.  In conclusion, we arrive at a new way of understanding the pure part of the P=W conjecture, as well as relating both sides of nonabelian Hodge theory to the theory of BPS state counting in noncommutative geometry.

\section*{Acknowledgements}
While writing this paper I was employed at the EPFL, supported by the Advanced Grant ``Arithmetic and physics of Higgs moduli spaces'' No. 320593 of the European Research Council.

\bibliographystyle{amsplain}
\bibliography{SPP}
\end{document}